\documentclass[12pt]{article}
\usepackage{geometry}
\usepackage{amsmath,amssymb}
\usepackage{graphicx}
\usepackage[colorlinks=true,citecolor=black,linkcolor=black,urlcolor=blue]{hyperref}

\usepackage[ruled]{algorithm2e} 

\SetAlFnt{\small}
\SetAlCapFnt{\small}
	\SetAlCapNameFnt{\small}
\SetAlCapHSkip{0pt}
\IncMargin{-1.6\parindent}

\usepackage{amsmath,amssymb,amsthm}
\usepackage{enumerate}
\usepackage{mathtools}
\usepackage{tikz}
\usepackage{url}

\usepackage{pifont}

\usetikzlibrary{positioning,chains,fit,shapes,calc}

\newcommand{\N}{\mathbb{N}}

\definecolor{myblue}{RGB}{80,80,160}
\definecolor{mygreen}{RGB}{80,160,80}

\usepackage{multirow}

\usepackage{thmtools,thm-restate}
\usepackage{cleveref}

\usepackage{ifthen}

\newtheorem{theorem}{Theorem}
\newtheorem{lemma}[theorem]{Lemma}

\newtheorem{proposition}[theorem]{Proposition}

\newtheorem{claim}[theorem]{Claim}

\newtheorem{definition}[theorem]{Definition}
\newtheorem{example}[theorem]{Example}

\newtheorem{remark}[theorem]{Remark}

\numberwithin{theorem}{section}
\numberwithin{lemma}{section}
\numberwithin{proposition}{section}
\numberwithin{remark}{section}
\numberwithin{example}{section}
\numberwithin{corollary}{section}
\numberwithin{figure}{section}
\numberwithin{definition}{section}

\title{Switch-based Markov Chains for Sampling Hamiltonian Cycles in Dense Graphs}

\author{Pieter Kleer\footnote{P. Kleer was supported by the Netherlands Organisation for Scientific Research
(NWO) through the Gravitation Programme Networks (024.002.003).}\\
\small Max Planck Institute for Informatics\\[-0.8ex]
\small Saarland Informatics Campus (SIC)\\[-0.8ex] 
\small Saarbrücken, Germany \\
\small\tt pkleer@mpi-inf.mpg.de\\
\and
Viresh Patel \qquad  Fabian Stroh \footnote{V. Patel and F. Stroh were supported by the Netherlands Organisation for Scientific Research
(NWO) through the Gravitation Programme Networks (024.002.003) and the NWO TOP
grant (613.001.601). }\\
\small University of Amsterdam\\[-0.8ex]
\small Korteweg-de Vries Institute (KdVI)\\[-0.8ex]
\small Amsterdam, The Netherlands\\
\small\tt \{vpatel,f.j.m.stroh\}@uva.nl}

\begin{document}

\maketitle

\begin{abstract}
We consider the irreducibility of switch-based Markov chains for the approximate uniform sampling of Hamiltonian cycles in a given undirected dense graph on $n$ vertices. As our main result, we show that every pair of Hamiltonian cycles in a graph with minimum degree at least $n/2+7$ can be transformed into each other by switch operations of size at most $10$, implying that the switch Markov chain using switches of size at most $10$ is irreducible. As a proof of concept, we also show that this Markov chain is rapidly mixing on dense monotone graphs.
\end{abstract}

\section{Introduction}
In this work, we consider the problem of sampling Hamiltonian cycles in dense graphs using switch-based Markov chains. Throughout, let $G$ be an $n$-vertex graph and denote its minimum degree by $\delta(G)$. A Hamiltonian cycle of $G$ is a simple cycle of $G$ that includes every vertex. A classical theorem by Dirac \cite{Dirac1952} states that if $\delta(G) \geq n/2$ then $G$ has a Hamiltonian cycle. Moreover, it is well known that in general it is NP-complete to decide if $G$ has a Hamiltonian cycle even if $\delta(G) \geq (\frac{1}{2} - \varepsilon) n$.

 Dyer, Frieze, and Jerrum \cite{Dyer1998} considered the question of counting and sampling Hamiltonian cycles in dense graphs. 
 They consider a Markov Chain Monte Carlo (MCMC) approach for solving the sampling problem. Here, one defines a suitable Markov chain on the (exponentially large) set of all Hamiltonian cycles, and shows that it is rapidly mixing, i.e., only a polynomial number of steps of the chain are needed in order to obtain a sample that is close to uniform. In particular, they give a \emph{fully-polynomial almost uniform sampler} for sampling Hamiltonian cycles from graphs $G$ with $\delta(G) \geq (\frac{1}{2} + \varepsilon)n$, which is then turned into a \emph{fully-polynomial randomised approximation scheme} for counting Hamiltonian cycles in such graphs by a standard reduction. 
 
 For the sampling problem, they take a two-step approach. First, based on a result of Jerrum and Sinclair \cite{Jerrum1990}, they show that there is a rapidly mixing Markov chain on the set of all $2$-factors of $G$ (which are all subgraphs of $G$ in which every vertex has degree $2$). Then it is shown that the number of $2$-factors in $G$ is at most a polynomial factor larger than the number of Hamiltonian cycles in $G$. This then automatically implies (roughly speaking) that if one takes a polynomial number of samples from the Markov chain that samples $2$-factors, most likely one of those samples will be a Hamiltonian cycle. This sample is then also an approximately uniform sample from the set of all Hamiltonian cycles in $G$.

At the end of their paper, Dyer, Frieze and Jerrum \cite{Dyer1998} ask if there is a rapidly mixing Markov chain on the set of Hamiltonian cycles, and possibly `near-Hamiltonian cyles', that mixes rapidly.\footnote{To be precise, in \cite{Dyer1998} they ask:  ``Second, is there a random walk on Hamilton cycles and (in some sense) “near-Hamilton cycles” which is rapidly mixing?''} As a first step towards addressing this question, we show there exist switch-based Markov chains on the set of all Hamiltonian cycles of a dense graph that converge to the uniform distribution, provided that $\delta(G) \geq \frac{1}{2}n + 7$. \\

\noindent Switch Markov chains are arguably the simplest and most natural Markov chains on the set of Hamiltonian cycles of a graph. Given a graph $G$, let $\mathcal{H}_G$ denote the set of Hamiltonian cycles of $G$. We say that $H'\in \mathcal{H}_G$ can be obtained from $H \in \mathcal{H}_G$ by a $k$-switch if $|E(H) \triangle E(H')| \leq 2k$, that is, a $k$-switch is an operation for transforming one Hamiltonian cycle into another by altering at most $2k$ of its edges.\footnote{Such operations are also widely used, for example, in heuristics for the travelling salesman problem see, e.g., \cite{Lin1973}.}

For a given constant $k \in \N$, the \emph{$k$-switch Markov chain} on $\mathcal{H}_G$ is defined as follows in this work. Given that the Markov chain is currently in state $H \in \mathcal{H}_G$, we first pick $\ell \in \{1,\dots,k\}$ uniformly at random, and then select a set $L \subseteq E(G)$ with $|L| = 2\ell$ uniformly at random. If the graph $H'$ with edge set
$$
E(H') = E(H) \triangle L
$$ 
is again in $\mathcal{H}_G$, i.e., a Hamiltonian cycle of $G$, then we transition to $H'$. Otherwise, we do nothing and stay in the state $H$.\footnote{There exist many algorithmic rules to select a switch.} See Figure~\ref{fig:switch_ex} for an example.  
\begin{figure}[ht!]
	\begin{center}
		\includegraphics[width=0.4\linewidth]{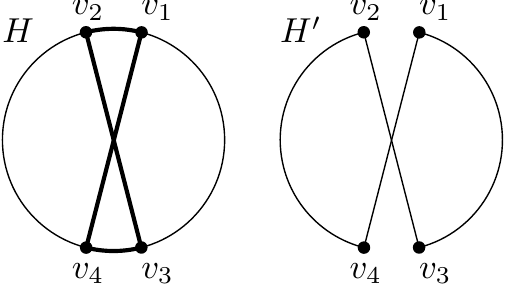}
		\caption{Example of a switch for $k=2$. Left side: The Hamiltonian cycle $H$ is the circle. We switch along the cycle $v_1v_2v_3v_4$, drawn thick. Right side: The modified graph $H'$, which is also a Hamiltonian cycle.}
		\label{fig:switch_ex}
	\end{center}
\end{figure}

It is not hard to show that the $k$-switch Markov chain will converge to the uniform distribution on $\mathcal{H}_G$ (because of symmetry of the transition probabilities) provided that the  chain is \emph{irreducible}. Irreducibility here refers to the fact that any two Hamiltonian cycles $H_1,H_2 \in \mathcal{H}_G$ can be transformed into each other by a sequence of $k$-switches.

\subsection{Our contributions}\label{sec:contributions}
The main goal of this work is to provide the first irreducibility results for the $k$-switch Markov chain. 
Given a graph $G$, we say $\mathcal{H}_G$ is \emph{$k$-switch irreducible} if for every $H, H' \in \mathcal{H}_G$, we can obtain $H'$ from $H$ by a sequence of $k$-switches, i.e., there exist $H_0, \ldots, H_r \in \mathcal{H}_G$ with $H_0 = H$ and $H_r = H'$ where $|E(H_i) \triangle E(H_{i+1})| \leq k$ for $i = 0, \ldots, r-1$. Our results are as follows. 
\begin{enumerate}[(i)]
\item We prove that $\mathcal{H}_G$ is $10$-switch irreducible if $\delta(G) \geq \frac{1}{2}n + 7$. \label{con:i} 
\item For each $k \geq 4$, we give examples of graphs $G$ satisfying $\delta(G) \geq \frac{n - 3k -4}{2}$ for which $\mathcal{H}_G$ is not $k$-switch irreducible.\label{con:ii}
\item We give examples of graphs $G$ with $\delta(G) \geq \frac{2}{3}n - 1$ for which $\mathcal{H}_G$ is not $2$-switch irreducible.\label{con:iii}
\end{enumerate}
The second item essentially establishes that, for the case $k = 10$,  the result in the first item is best possible (up to a constant-sized gap between $n/2 - 17$ and $n/2+7$).
Moreover, the third item shows that the $2$-switch Markov chain (probably the simplest Markov chain on Hamiltonian cycles) cannot be used to address the question of Dyer, Frieze and Jerrum for all dense graphs with $\delta(G) \geq n/2$.

As a proof of concept, we  show that, for dense monotone graphs $G$ (discussed in the related work section), the $k$-switch Markov chain on $\mathcal{H}_G$ is rapidly mixing. We do this by means of a meta-theorem which shows that if the $k$-switch Markov chain on $\mathcal{H}_G$ is (strongly) irreducible for some dense monotone graph $G$, then it is also rapidly mixing. 
 Here, strong irreduciblity roughly refers to the fact that if two Hamiltonian cycles are close to each other in terms of symmetric difference, we should be able to transform them into each other using a small number of $k$-switches; a formal definition is given later on. In the first item above, we show indeed this strong version of irreducibility for $k = 10$. 
 
\noindent Overall, several interesting new questions arise in light of our work and we hope our results will stimulate more work in the area. In particular, what is the smallest $k$ for which the $k$-switch Markov chain is (strongly) irreducible for dense graphs with $\delta(G) \geq \frac{n}{2} + c$, where $c$ is a (small) constant? Furthermore, given the vast interest in the $2$-switch Markov chain for other combinatorial objects (see Section \ref{sec:related_work}), what is the smallest\footnote{It is not hard to argue that the result is true for for complete graphs $G$ where $\gamma = c = 1$.} constant $\frac{2}{3} <\gamma < 1$ such that the $2$-switch Markov chain is irreducible for all dense graphs with $\delta(G) \ge \gamma n + c$ for some (small) constant $c$?

\subsection{Related work}\label{sec:related_work}
The question of irreducibility, as well as being integral to the MCMC method, is studied in its own right under the moniker of reconfiguration problems. Here, one wishes to decide whether the space of solutions to some combinatorial problem is connected (where two solutions are adjacent if one can be obtained from the other by some small prescibed change); see for example the surveys of van den Heuvel \cite{vdHeuvel2013} and Nishimura \cite{Nishimura2018}. 
Reconfiguration problems about Hamiltonian cycles have not been widely considered. Takaoka \cite{Takaoka2018} has considered the complexity of deciding whether $\mathcal{H}_G$ is  $2$-switch irreducible when $G$ belongs to particular structural graph classes. This includes a hardness result for chordal bipartite graphs, but also a result establishing the $2$-switch irreducibility of Hamiltonian cycles in unit interval graphs and  monotone graphs.
 A slightly different Hamiltonian reconfiguration problem is considered by Lignos \cite{Lignos2017}.

The mixing time of switch-based Markov chains have been studied extensively for sampling subgraphs of $K_n$ with a given degree sequence, see, e.g., \cite{Kannan1999,Cooper2007,Erdos2013,AK2019}. It is well known, see e.g. \cite{Taylor1981}, that every two graphs (thought of as subgraphs on $K_n$) with the same degree sequence can be transformed into each other with switches of size $2$ (in $K_n$). This remains true if one restricts oneself to the class of all connected subgraphs of $K_n$ with a fixed degree sequence \cite{Taylor1981}. In particular, relevant to our setting, Feder et al. \cite{Feder2006} (implicitly) show that the $2$-switch chain is rapidly mixing on the set of all Hamiltonian cycles in case $G$ is the complete graph. There are more direct ways to obtain this result, but we mention it here as we rely on some of their ideas in order to address the mixing time of the switch Markov chain on dense monotone graphs. 

Monotone graphs, also known as bipartite permutation graphs, have been widely studied from the structural graph theory perspective, perhaps most notably in their characterisation \cite{Spin} (we  define them formally in Section~\ref{sec:DMG}).
Monotone graphs have also been considered in the context of switch-based Markov chains for the sampling of perfect matchings: in particular, Dyer, Jerrum and M\"uller \cite{Dyer2017} show that the $2$-switch Markov chain for sampling perfect matchings is rapidly mixing on monotone graphs. We refer the reader to \cite{Dyer2017} for further results in this direction.

We mentioned above that Takaoka \cite{Takaoka2018} shows that the set of all Hamiltonian cycles in a given monotone graph is $2$-switch irreducible. We remark that this is established in the weak sense by showing that every Hamiltonian cycle can be transformed, by switches of size $2$, into a fixed \emph{canonical} Hamiltonian cycle. However, we need the stronger notion of irreducibility for our rapid mixing proof for dense monotone graphs to go through.

\section{Preliminaries}
Let $G = (V,E)$ be a simple undirected graph with vertex set $V = \{v_1,\dots,v_n\}$ and edge set $E = \{e_1,\dots,e_m\}$. We use the shorthand notation $uv$ to denote an edge $\{u,v\} \in E$. A $2$-factor of $G$ is a subgraph $F$ in which every vertex $v \in V$ has degree precisely $d_F(v) = 2$. We use $\mathcal{F}_G$ to denote the set of all $2$-factors of $G$. A Hamiltonian cycle is a connected $2$-factor, i.e., a simple cycle passing through all vertices of $G$. We use $\mathcal{H}_G$ to denote the set of all Hamiltonian cycles of $G$. Given two graphs $G=(V,E)$ and $G'=(V, E')$ on the same vertex set $V$, their symmetric difference  is denoted by $G \triangle G' = (V, E \triangle E') = (V, (E \setminus E') \cup (E'\setminus E)).$ 
We use $N_G(v) = \{ w : vw \in E \}$ to denote the set of neighbours of $v \in V$ in $G$ and we write $d_G(v) = |N(v)|$ for the degree of $v$ dropping subscripts when the graph is clear. 

For a given $k \geq 2$ and (finite) set $\mathcal{A}$ of graphs on some vertex set $V$, a switch of size $k$ (with respect to $\mathcal{A}$) is an operation on a given graph $F = (V,E) \in \mathcal{A}$ in which we remove exactly $k$ edges from $F$ and add exactly $k$ edges from the complement of $F$ in such a way that the resulting graph $F'$ also satisfies $F' \in \mathcal{A}$. We then define a $k$-switch to be a switch of size at most $k$.
 In this work we are mostly interested in $\mathcal{A} = \mathcal{H}_G$ or $\mathcal{A} = \mathcal{F}_G$ for a given undirected graph $G$.

Fix a graph $G$ and consider a $k$-switch with respect to $\mathcal{F}_G$ or $\mathcal{H}_G$. Writing $A$ for the $2k$ edges involved in the $k$-switch, it is easy to see that every vertex of the graph $S=(V,A)$ must have even degree (since all graphs in $\mathcal{F}_G$ or $\mathcal{H_G}$ are regular of degree $2$). Moreover, every connected component of $S$ can be thought of as an alternating circuit, i.e.\ a circuit whose edges alternate between edges in $G$ and edges not in $G$.\footnote{Recall that a circuit in $G=(V,E)$ is a sequence of $v_1e_1v_2e_2 \cdots v_{k-1}e_{k-1}v_k$ of vertices and edges where $e_i = v_iv_{i+1} \in E$, the edges $e_i$ are distinct, and $v_1=v_k$.}\\

\noindent \emph{$k$-Switch irreducibility.} For a given graph $G$ and integer $k$, we say that $\mathcal{H}_G$ is (weakly) $k$-switch irreducible if for every $H_1, H_2 \in \mathcal{H}_G$, there exists a sequence $H_1 = Z_1,\dots,Z_q = H_2
$ of Hamiltonian cycles in $\mathcal{H}_G$
such that every consecutive pair of Hamiltonian cycles $(Z_i,Z_{i+1})$ differs by a $k$-switch. 
Moreover, 
for a given class of graphs $\mathcal{G}$ and integer $k$, 
we say that $\mathcal{G}$ is \emph{strongly $k$-switch irreducible for Hamiltonian cycles} if there exists a function $\phi : \N \rightarrow \N$ with the following property: for all $G \in \mathcal{G}$, whenever  $H_1, H_2 \in \mathcal{H}_G$ with $|E(H_1)\triangle E(H_2)| \leq x$, there exists a sequence of Hamiltonian cycles $H_1 = Z_1,\dots,Z_q = H_2
$ of Hamiltonian cycles in $\mathcal{H}_G$
such that every consecutive pair of Hamiltonian cycles $(Z_i,Z_{i+1})$ differs by a $k$-switch operation and  $q \leq \phi(x)$. 

Roughly speaking, strong irreducibility states that if two Hamiltonian cycles are somewhat `close' to each other in terms of symmetric difference, then we should be able to transform one into the other with a `small' number of $k$-switches. Similarly we define (strong) irreducibility for $2$-factors. 

For a given graph $G$, we say that $\mathcal{F}_G$ (the set of $2$-factors of $G$) is (weakly) $k$-switch irreducible if for every $F_1, F_2 \in \mathcal{F}_G$, there exists a sequence $F_1 = Z_1,\dots,Z_q = F_2$ of $2$-factors in $\mathcal{F}_G$
such that every consecutive pair of $2$-factors $(Z_i,Z_{i+1})$ differs by a $k$-switch. 
For a given class of graphs $\mathcal{G}$ and integer $k$, 
we say that $\mathcal{G}$ is \emph{strongly $k$-switch irreducible for $2$-factors} if there exists a function $\phi : \N \rightarrow \N$ with the following property: for all $G \in \mathcal{G}$, whenever  $H_1, H_2 \in \mathcal{H}_G$ with $|E(H_1)\triangle E(H_2)| \leq x$, there exists a sequence $H_1 = Z_1,\dots,Z_q = H_2
$ of Hamiltonian cycles in $\mathcal{H}_G$
such that every consecutive pair of Hamiltonian cycles $(Z_i,Z_{i+1})$ differs by a $k$-switch operation and  $q \leq \phi(x)$.\\
\medskip

\noindent 
\emph{Markov chains and mixing times.} The preliminaries here will only be required in Section~\ref{se:rapid mixing} onwards. We write $\mathcal{M} = (\Omega,P)$ to denote an aperiodic, irreducible and time-reversible Markov chain $\mathcal{M}$ on state space $\Omega$ with transition matrix $P$.  We write $P^t(x,\cdot)$ for the distribution over $\Omega$ at time step $t$ given that the initial state is $x \in \Omega$. 
The \emph{total variation distance} of this distribution from the (unique) stationary distribution $\pi$ at time $t$ with initial state $x$ is 
\[
\Delta_x(t) = \max_{S \subseteq \Omega} \big| P^t(x,S) - \pi(S)\big| = \frac{1}{2}\sum_{y \in \Omega} \big| P^t(x,y) - \pi(y)\big| \,,
\]
and the \emph{mixing time} of $\mathcal{M}$ is 
\[
\tau(\epsilon) = \max_{x \in \Omega}\min\{ t : \Delta_x(t') \leq \epsilon \text{ for all } t' \geq t\}. 
\]
Informally, $\tau(\epsilon)$  is the number of steps until the Markov chain is $\epsilon$-close to its stationary distribution. When $\pi$ is the uniform distribution over $\Omega$, we say that a Markov chain is to be \emph{rapidly mixing} if the mixing time can be upper bounded by a function polynomial in $\ln(|\Omega|/\epsilon)$. 

As the Markov chains we consider are time-reversible, the matrix $P$ only has real eigenvalues, that we denote by $1 = \lambda_0 > \lambda_1 \geq \lambda_2 \geq \dots \geq \lambda_{|\Omega|-1} > -1$.  We can always replace the transition matrix $P$ of the Markov chain by $(P+I)/2$, to make the chain \emph{lazy}, and, hence,  guarantee that all its eigenvalues are non-negative. It then follows that the second-largest eigenvalue of (the new transition matrix) $P$ is $\lambda_1$. In this work we always consider the lazy versions of the Markov chains involved, but we do not always mention this explicitly.  
It follows directly from Proposition 1 in \cite{Sinclair1992} that
$$\tau(\epsilon) 
\leq  \frac{1}{1 - \lambda_1} \big(\ln(1/\pi_*) + \ln(1/\epsilon)\big),
$$
where $\pi_* = \min_{x \in \Omega} \pi(x)$. When $\pi$ is the uniform distribution, the above bound reduces to
$$\tau(\epsilon) 
\leq \frac{1}{1 - \lambda_1}(\ln(|\Omega|) + \ln(1/\epsilon)).
$$
The quantity $(1 - \lambda_1)^{-1}$ can be upper bounded using the \emph{multicommodity flow method} of Sinclair  \cite{Sinclair1992}.

We define the state space graph of the chain $\mathcal{M}$ as the directed graph $\mathbb{G}$ with vertex set $\Omega$ that contains exactly the arcs $(x,y) \in \Omega \times \Omega$ for which $P(x,y) > 0$ and $x \neq y$. Let $\mathcal{P} = \cup_{x \neq y} \mathcal{P}_{xy}$, where $\mathcal{P}_{xy}$ is the set of simple paths between $x$ and $y$ in $\mathbb{G}$.
A \emph{flow} $f$ in $\Omega$ is a function $\mathcal{P} \rightarrow [0,\infty)$ with the property 
$\sum_{p \in \mathcal{P}_{xy}} f(p) = \pi(x)\pi(y)$ for all $x,y \in \Omega, x \neq y$.
The flow $f$ can be extended to a function on oriented edges of $\mathbb{G}$ by setting 
$f(e) =  \sum_{p \in \mathcal{P} : e \in p } f(p)$,
so that $f(e)$ is the total flow routed through the edge $e \in E(\mathbb{G})$. Let $\ell(f) = \max_{p \in \mathcal{P} : f(p) > 0} |p|$ be the length of a longest flow carrying path, and let 
$
\rho(e) = f(e)/Q(e)
$
be the \emph{load} of the edge $e$, where $Q(e) = \pi(x)P(x,y)$ for $e = (x,y)$.
The maximum load of the flow is then given by 
$
\rho(f) = \max_{e\in E(\mathbb{G})} \rho(e).
$ 
Sinclair, in Corollary $6^{\,\prime}$ of \cite{Sinclair1992}, shows that 
$$
(1 - \lambda_1)^{-1} \leq \rho(f)\ell(f).
$$ 

We  use the following (by now standard) technique for bounding the maximum load of a flow in case the chain $\mathcal{M}$ has uniform stationary distribution $\pi$. 
Suppose $\theta$ is the smallest positive transition probability of the Markov chain between two distinct states in $\Omega$.
If $b$ is such that $f(e) \leq b / |\Omega|$ for all $e\in E(\mathbb{G})$, then it follows that
$
\rho(f) \leq b/\theta
$.
This implies that
\begin{equation*}
\tau(\epsilon) \leq \frac{\ell(f)\cdot b}{\theta}\ln(|\Omega|/\epsilon)\,.
\end{equation*}
Now, if $\ell(f), b$ and $1/\theta$ can be bounded by a function polynomial in $\ln(|\Omega|)$, it follows that
the Markov chain $\mathcal{M}$ is rapidly mixing. In this case, we say that $f$ is an \emph{efficient} flow.
Note that in this approach the transition probabilities do not play a role as long as $1/\theta$ is polynomially bounded.

\section{Irreducibility of $k$-switch Markov chain}
In this section we will prove various results regarding the (non)-irreducibility of the $k$-switch Markov chain. The main result of this section is Theorem \ref{thm:main_irr} below. Afterwards, we provide various examples of non-irreducibility for certain combinations of $\delta(G)$ and $k$.

\begin{theorem}\label{thm:main_irr}
	If a graph $G$ satisfies $\delta(G) \geq \frac{1}{2}n + 7$, then the set $\mathcal{H}_G$ of all Hamiltonian cycles of $G$ is $10$-switch irreducible. Moreover, the class of graphs $G$ for which $\delta(G) \geq \frac{1}{2}n + 7$ is strongly $10$-switch irreducible for Hamiltonian cycles. 
\end{theorem}

\begin{remark}[Bipartite case]\label{rem:bipartite}
Theorem \ref{thm:main_irr} remains true if we restrict ourselves to bipartite graphs $G = (A \cup B, E)$, where $|A| = |B| = n$, and $\delta(G) \geq \frac{1}{2}n + 7$. The proofs are almost identical, so we make remarks in footnotes where the proofs differ.
\end{remark}

In order to prove Theorem \ref{thm:main_irr}, we rely on the following lemma. It allows us to quickly reconfigure a $2$-factor $T$ into a Hamiltonian cycle $H'$ without increasing the symmetric difference with respect to some fixed Hamiltonian cycle $H$.

\begin{lemma}[Reconnecting lemma]\label{lem:reconnect}
Let $G = (V,E)$ be an undirected graph with minimum degree $\delta(G) \geq \frac{1}{2}n + 1$,  
and let $H$ be a fixed Hamiltonian cycle in $G$. 
Let $T$ be an arbitrary $2$-factor of $G$ with $t$ components. 

Then there exists a Hamiltonian cycle $H'$, so that $T$ can be transformed into $H'$ with at most $t-1$ switches of size at most $3$, and for which
\begin{equation}\label{eq:symm_decrease}
|H' \triangle H| \leq |T \triangle H|.
\end{equation}
\end{lemma}

\begin{proof}
	Let $t$ be the number of components of $T$. We will prove the statement in the lemma using induction. If $t = 1$ then $T$ is Hamiltonian and we are done as we may take $H' = T$. Suppose $t > 1$. Let $C_1,\dots,C_t$ denote the cyclic components of $T$. Since $H$ is Hamiltonian, there must be some edge $vw \in E(H)$ connecting two components of $T$ (see Figure~\ref{fig:reconn1}). We assume without loss of generality that $vw$ connects $C_1$ and $C_2$, i.e\ that $v \in V(C_1)$ and $w \in V(C_2)$ (by renumbering if necessary).	Moreover, since $v$ has degree two in $H$ and $vw \in E(H)$, it must be that there exists an $a \in V(C_1)$ (one of the two neighbours of $v$ in $T$) so that $va \in E(T)$, but $va \notin E(H)$. Similarly, there is a $b \in V(C_2)$ so that $wb \in E(T)$, but $wb\notin E(H)$.

	We assign orientations to $C_1, \dots, C_t$. We will call $v^+$ the vertex following $v$ in the appropriate orientation and $v^-$ the vertex preceding $v$. We choose the orientations on $C_1$ and $C_2$ such that $v = a^+$ and $b = w^+$, see Figure~\ref{fig:reconn1}, 
	\begin{figure}[ht!]
		\begin{center}
			\includegraphics[width=0.75\linewidth]{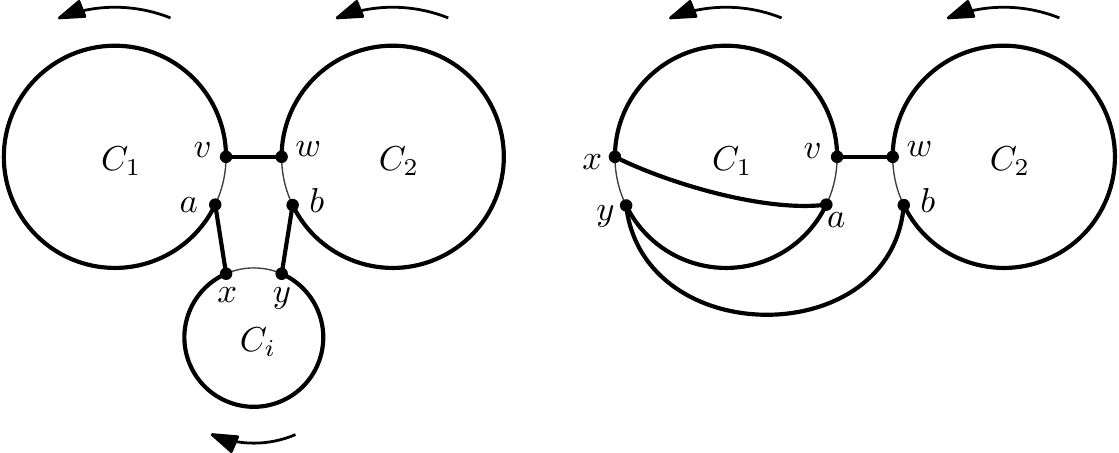}
			\caption{Two situations in the general case. The thick black line shows the cycle after the switch. Left: $xy$ is on a third cycle. Right: $xy$ is on $C_1$.}
			\label{fig:reconn1}
		\end{center}
	\end{figure}
	and we assign arbitrary orientations on $C_3, \dots, C_t$. Consider $X:=\{v^+\mid v\in N(a)\}$. As $\delta(G)\geq \frac{1}{2}n+1$, $|X|\geq \frac{1}{2}n+1$. Also consider $N(b)$, and note that we have $|N(b)|\geq \frac{1}{2}n+1$. Therefore $|X\cap N(b)|\neq \emptyset$. Select $y\in X\cap N(b)$ and set $x=y^-$ noting that $ax\in E(G)$.\footnote{In the case of bipartite graphs (see Remark~\ref{rem:bipartite}), we note that $avwb$ is a path of $G$ so $a$ and $b$ are in different parts, say $a \in A$ and $b\in B$. Then $X \subseteq A$ with $|X| \geq \frac{1}{2}n +1$ and $N(b) \subseteq A$ with $|N(b)| \geq \frac{1}{2}n + 1$, so $X \cap N(b) \not= \emptyset$ and we continue.} 
	If $y\notin \{a,b^+,w,v^+\}$, the general case, we now switch along the cycle $vaxybwv$; see Figure~\ref{fig:reconn1}. Note that the edge $xy$ may lie on $C_1$, $C_2$ or a different cycle $C_i$. In all these cases, we do not increase $|T\Delta H|$, as $vw \in E(H)$ and $va, bw \notin E(H)$. If $xy \notin E(C_1\cup C_2)$, we decrease the number of cycles by two, otherwise by one.
	\begin{figure}		
		\begin{center}
			\includegraphics[width=0.75\linewidth]{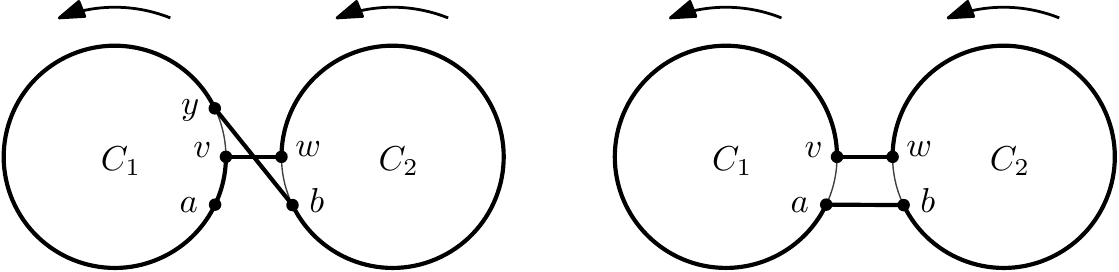}
			\caption{Two situations from the special cases. Left: $y=v^+$, Right: $y=a$, $y=b^+$}
			\label{fig:reconn2}
		\end{center}
	\end{figure}
	For the special cases $y\in \{a,b^+,w,v^+\}$, we switch along different cycles as follows; see Figure~\ref{fig:reconn2}.
	If $y=v^+$, we switch along the cycle $vybwv$. If $y=w$, we switch along the cycle $vaxwv$.
	If $y\in \{a,b^+\}$, then $ab \in E(G)$, and we switch along the cycle $vabwv$. 
	It is easy to see that in these cases we decrease $|T\Delta H|$ by at least two and we decrease the number of cycles by one.
	
	In any case, the resulting $2$-factor has fewer components and the symmetric difference is not larger. Repeated application of this procedure proves the statement of the lemma.
\end{proof}

We now continue with the proof of Theorem \ref{thm:main_irr}.

\begin{proof}[Proof of Theorem \ref{thm:main_irr}]
	We claim that for two given Hamiltonian cycles $H_1$ and $H_2$ there is  a switch of size at most $4$ that transforms $H_1$ into a $2$-factor $T$ with at most $3$ components such that $|T \triangle H_2| < |H_1 \triangle H_2|$.
The theorem then follows from Lemma~\ref{lem:reconnect} since with two switches of size at most $3$, we can transform $T$ into some Hamiltonian cycle $H'$ satisfying 
\[
|H' \triangle H_2| \leq |T \triangle H_2| < |H_1 \triangle H_2|.
\]
In particular we can transform $H_1$ to $H'$ with a switch of size at most $4 + 2 \times 3 = 10$, and repeating this we can transform $H_1$ into $H_2$ with at most $x = |H_1 \triangle H_2|$ switches of size $10$, proving the theorem (where we take $\phi(x)=x$ in the definition of strong irreducibility). 

We now prove the claim. Note that the symmetric difference of $H_1$ and $H_2$ is the vertex-disjoint union of circuits in which edges alternate between $H_1$ and $H_2$ and the circuits visit each vertex zero, one, or two times. If the symmetric difference of $H_1$ and $H_2$ contains such alternating circuits with four or six edges (corresponding to switches of size $2$ or $3$), the claim obviously holds, so assume otherwise. 
In this case it is not hard to see that we can find an $H_1,H_2$-alternating walk 
 $P = a_1a_2a_3a_4a_5a_6$ (here the $a_i$ are vertices and $a_1 $ and $a_6$ are distinct)  
 such that the $a_1a_2, a_3a_4, a_5a_6$ are edges of $H_1$, and $a_2a_3$, $a_4a_5$ are edges of $H_2$.

	We try to find vertices $b$ and $c$ that are neighbours on $H_1$ such that $b\in N(a_1)$ and $c \in N(a_6)$. Then the circuit $C:=a_1a_2a_3a_4a_5a_6cba_1$ is a $4$-switch for $H_1$. Deleting the edges $a_1a_2$, $a_3a_4$, $a_5a_6$ and $cb$ divides $H_1$ into four paths and adding $a_2a_3$, $a_4a_5$, $a_6c$ and $ba_1$ can connect some of these paths again. 

Therefore, switching $H_1$ along $C$ can produce at most $4$ connected components, and this only happens if the four edges $a_2a_3$, $a_4a_5$, $a_6c$ and $ba_1$ connect each path into a cycle (see Figure~\ref{fig:corr1}, left side). If one of the paths is just an isolated vertex, it cannot be connected to itself in this way. It is easy to check that $4$ components are produced if and only if the vertices $a_1, a_2, \ldots, a_6, c, b$ are distinct and appear in that order along $H_1$ (as in Figure~\ref{fig:corr1}, left side).
	\begin{figure}[ht!]
		
		\begin{center}
			\includegraphics[width=0.6\linewidth]{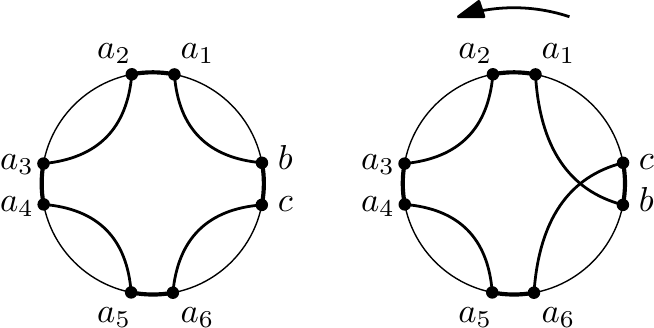}
			\caption{Left side: The circle is $H_1$. The only way that a $4$-switch (thick lines) leads to four components is the shown configuration. Right side: Choosing $c$ to follow $b$ leads to at most three components. Note that in general the edges $a_3a_4$ and $a_5a_6$ could appear in different places and orientations.}
			\label{fig:corr1}
		\end{center}
	\end{figure}
	To prevent this, we choose $b$ and $c$ as follows: orient $H_1$ so that $a_2$ follows $a_1$. We call the vertex following a vertex $v$ in this orientation $v^+$ and the previous vertex $v^-$. Set $M=\{v^-: v\in N(a_6)\}$ and consider $N(a_1)\cap M$.  As both $|N(a_1)|,|M|\geq n/2+7$ we have $|N(a_1)\cap M|\geq 2 \cdot 7 = 14$.\footnote{In the case of bipartite graphs (see Remark~\ref{rem:bipartite}), we note that $a_1a_2a_3a_4a_5a_6$ is a walk in $G$ and so $a_1$ and $a_6$ are in different parts; say $a_1 \in A$ and $a_6 \in B$. Then $N(a_1), M \subseteq B$, so since $|N(a_1)|, |M| \geq \frac{1}{2}n+7$, so $|N(a_1) \cap M| \geq 14$, and we continue as before.}
	 Select $b \in (N(a_1)\cap M)\setminus \{a_i^+,a_i^-, i=1,\dots,6\}$ and set $c=b^+$.
	 This ensures that the resulting $4$-switch (along the circuit $C:=a_1a_2a_3a_4a_5a_6cba_1$) produces at most three components.
	
	Finally, if $T$ is the $2$-factor produced by switching $H_1$ along $C$, then compared to $H_1$, $T$ contains at least two new edges of $H_2$  (namely $a_2a_3, a_4a_5$) but $T$ may have lost one edge of $H_2$ (namely $bc$ if it was in fact an edge of $H_2$), giving a net gain of one. Since $T$ and $H_1$ have the same number of edges, we see that $|T \triangle H_2| \leq |H_1 \triangle H_2| - 1$, as required.
\end{proof}

We also give a version of Theorem~\ref{thm:main_irr} for $2$-factors, instead of Hamiltonian cycles, that we will need later. The proof is a simplification of Theorem~\ref{thm:main_irr} and so we defer its proof to the appendix.

\begin{proposition}
\label{pr:2-factor}
The class of graphs $G$ for which $\delta(G) \geq \frac{1}{2}n + 7$ is strongly $4$-switch irreducible for $2$-factors.
	
	For bipartite graphs the following holds.
	The class of bipartite graphs $G = (A \cup B, E)$  with bipartition $A \cup B$, where $|A| = |B| = n$, and $\delta(G) \geq \frac{1}{2}n + 7$ is strongly $k$-switch irreducible for $2$-factors. 
\end{proposition}

We continue with examples showing non-irreducibility under certain assumptions on $\delta(G)$ and $k$, as stated in contributions \eqref{con:ii} and \eqref{con:iii} in Section \ref{sec:contributions}.

\begin{example}[The case $\delta(G) = \frac{2n}{3}-1$ and $k = 2$]
	Construct $G=(V,E)$ as follows: Set $V = A_1\cup A_2 \cup A_3$, where $|A_i| = n/3=:m$. For convenience, we select $n$ such that $m$ is odd and $m\geq 3$. We denote the vertices of $A_i$ by $v_{i,j}$ for $j = 1,\dots, m$.
 Take as edge set $E$ all edges between vertices in $A_1$, all edges between vertices in $A_3$, and all edges from vertices in $A_i$ to vertices in $A_{i+1}$ for $i=1,2$ (see Figure \ref{fig:nonIrr1}).
	
	We color edges as follows: All edges incident to a vertex in $A_1$ are colored blue, and all other edges red.
	Note that all cycles of length 4 contain an even number of red and blue edges.
	This means that any switch along a 4-cycle preserves the parity of red and blue edges.
	
	We will finish the construction by describing two Hamiltonian cycles $H_1$ and $H_2$ that have different parities of blue edges. As any 2-switches preserve the parity of blue edges, $H_1$ cannot be converted to $H_2$ via 2-switches.
	
	The blue edges in $H_1$ are $v_{2,1}v_{1,1}$, $v_{1,k}v_{1,k+1}$ for $k=1,\dots,m-1$ and $v_{1,m}v_{2,m}$. The red edges in $H_1$ are $v_{2,k}v_{3,k}$, $v_{3,k}v_{2,k+1}$ for $k=1,m-2$ and $v_{2,m-1}v_{3,m-1}$, $v_{3,m-1}v_{3,m}$, $v_{3,m}v_{2,m}$. There are an even number of blue edges and an odd number of red edges in $H_1$.
	The Hamiltonian cycle $H_2$ is constructed by swapping the roles of the blue and red edges.

	\begin{figure}[ht!]
		
		\begin{center}
			\includegraphics[width=0.75\linewidth]{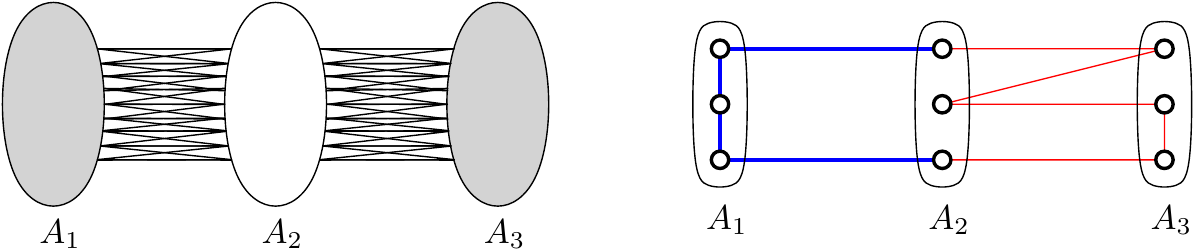}
			\caption{Left: The graph $G$. Right: The Hamiltonian cycle $H_1$ in $G$ with $n=9$. There are an even number of (thick) blue edges and an odd number of (thin) red edges.}
			\label{fig:nonIrr1}
		\end{center}
	\end{figure}	
\end{example}
	
\begin{example}[The case $\delta(G) \approx \frac{n}{2}$ for each fixed $k$.]
For $k$ fixed and $n \geq 3k+5$, there is a graph $G$ with $\delta(G) \geq (n - 3k -4)/2$ for which $\mathcal{H}_G$ is not $k$-switch irreducible.
	Our construction relies on the following lemma.
	\begin{lemma}		
		For any $\ell$, there is a graph $X$ with  $3\ell+1$ vertices that has exactly two Hamiltonian paths $H_1$ and $H_2$. Moreover, these two paths satisfy $|H_1\Delta H_2| = 2\ell$.
	\end{lemma}
	\begin{proof}
		Without loss of generality let $\ell$ be odd, and set $n=3\ell+1$. Let $X = (V,E)$ with $V := \{v_1,\dots ,v_n\}$ and $E:= E_1\cup E_2$, where $E_1 = \{v_iv_{i+1} \mid 1\leq i\leq n-1\}$, and $E_2 =\{ v_jv_{j+4} \mid j \equiv 2(\text{mod }3) \text{ and } j \leq n-5\} \cup \{v_3v_{n-2}\}$; see left side of Figure~\ref{fig:nonIrr3}.
		As vertices $v_1$ and $v_n$ have degree $1$, they must be the ends of any Hamiltonian path in $X$. Vertices $v_i$ with $i \equiv 1(\text{mod }3)$ and $4\leq i \leq n-3$ have degree $2$ in $X$, so both of their incident edges must be part of any Hamiltonian path; call the set of these $2\ell$ edges $F$ and call the remaining edges $F'$. Note that the edges of $F'$ form a cycle $C \subseteq X$. In $F$, every vertex of $V$ has degree $1$ or $2$ and those vertices of degree $1$ (except for $v_1$ and $v_n$) are precisely the vertices in the cycle $C$. Therefore we can only extend $F$ to a Hamiltonian path by adding a perfect matching from $C$, and it is easy to see that adding either perfect matching from $C$ results in a Hamiltonian path. These Hamiltonian paths have symmetric difference of size $|E(C)| = |F'| = 2 \ell$.
	\end{proof}
	\begin{figure}[ht!]
		
		\begin{center}
			\includegraphics[width=0.6\linewidth]{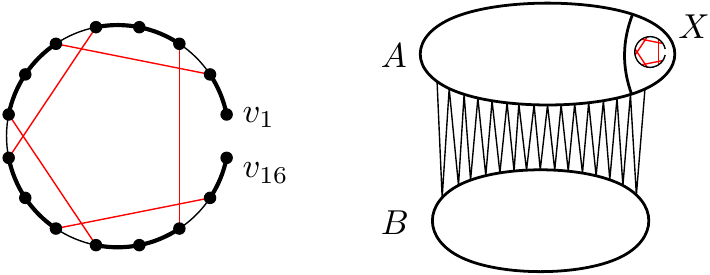}
			\caption{Left: Example of $X$ for $\ell=5$ (edges in $E_1$ are black, edges in $E_2$ are red, edges in $F$ are heavy, edges in $C$ are thin). Right: The example graph $G$ }
			\label{fig:nonIrr3}
		\end{center}
	\end{figure}
	
For the example we begin by applying the previous lemma with $\ell=k+1$ to obtain the graph $X$ of order $r:=3\ell + 1$. For any $n$ such that $n+r$ is odd, we construct our example $G$ by taking an (unbalanced)	complete bipartite graph  with parts $A$ and $B$ of size $\frac{n+(r-1)}{2}$ and $\frac{n-(r-1)}{2}$ respectively and adding a copy of $X$ inside $A$. See Figure~\ref{fig:nonIrr3}, right side.
	
As there are no edges inside $B$, any Hamiltonian cycle of $G$ must use $r-1$ edges inside $A$, and so these must be within $X$. Since $X$ has $r$ vertices, any Hamiltonian cycle of $G$ must induce a Hamiltonian path on $X$. By construction, $X$ has exactly two Hamiltonian paths $H_1$ and $H_2$, and they have a symmetric difference of $2k+2$. It is easy to see that $G$ has Hamiltonian cycles that use each of the two Hamiltonian paths in $X$, but it is impossible to perform a sequence of $k$-switches to transform a Hamiltonian cycle that uses $H_1$ into one that uses $H_2$; indeed if such a sequence existed, examining its restriction to $X$ would yield a sequence of switches of size at most $k$ that transforms $H_1$ into $H_2$ but maintaining a Hamiltonian path in $X$ at each stage; this is impossible since $X$ has only two Hamiltonian paths and their symmetric difference has size $2 \ell = 2(k+1) > k$.
\end{example}

\section{Rapid mixing for dense monotone graphs}
\label{se:rapid mixing}
In this section we will give some rapid mixing results for switch-based Markov chains for sampling Hamiltonian cycles in special classes of dense graphs. 

We first present a result for the sampling of $2$-factors using switch-based Markov chains, which will be used later on, and that might be of independent interest.
Given a graph $G$, recall the $k$-switch Markov chain on $\mathcal{H}_G$ defined in the introduction. Replacing $\mathcal{H}_G$ with $\mathcal{F}_G$ (the set of all $2$-factors of $G$) everywhere in that definition defines the $k$-switch Markov chain on $\mathcal{F}_G$.

\begin{theorem}\label{thm:2_mixing}
Let $\mathcal{G}$ be the class of all graphs $G$ on $n$ vertices with $\delta(G) \geq n/2$.  If $\mathcal{G}$ is strongly $k$-switch irreducible for $2$-factors for some $k \in \mathbb{N}$ (this is the case for $k=4$ by Proposition~\ref{pr:2-factor}) then there is an efficient multicommodity flow for the $k$-switch Markov chain on $\mathcal{F}_G$, and, in particular, this Markov chain is rapidly mixing.
\end{theorem}
Moreover, Theorem \ref{thm:2_mixing} remains true for the bipartite case of the problem, where we are given a bipartite graph $G = (A \cup B, E)$ with both $|A| = |B|= n$, and where every vertex in $A \cup B$ has degree at least $n/2$.

The proof of Theorem~\ref{thm:2_mixing} is outlined in Appendix \ref{app:2_mixing}.
It is based on the embedding argument introduced in  \cite{AK2019} for the switch Markov chain that samples graphs with a given degree sequence. It is perhaps interesting to note that it seems much harder to prove Theorem \ref{thm:2_mixing} by using other approaches for that problem, such as \cite{Cooper2007,Erdos2013}. These approaches do have the advantage that they get better mixing time bounds than those in \cite{AK2019}.

\subsection{Dense monotone graphs}
\label{sec:DMG}
In this section we will describe a rapid mixing result for sampling Hamiltonian cycles from dense monotone graphs that is based on Theorem \ref{thm:2_mixing}. We will start with the definition of monotone graphs (also known as bipartite permutation graphs).

\begin{definition}[Monotone graph]
\label{def:monotone}
A bipartite graph $G = (A\cup B, E)$, with $|A| = |B| = n$, is \emph{monotone} if there exists a permutation $(a_1,\dots,a_n)$ of the vertices in $A$ and a permutation $(b_1,\dots,b_n)$  of the vertices in $B$, such that the adjacency matrix $C$ of $G$, with rows indexed by $a_1,\dots,a_n$ and columns indexed by $b_1,\dots,b_n$, has \emph{monotone rows and columns}. 
This means that for each $i$, there exists $1 \leq r_i \leq t_i \leq n$ such that $C(a_i, b_j) =1$ if and only if  $r_i \leq j \leq t_i$ and the sequences $(r_i)_{i=1}^n$ and $(t_i)_{i=1}^n$ are non-decreasing. Intuitively, this means that the $1$-entries in every row and column are contiguous.
 Note that although the definition does not immediately  appear to be symmetric in $A$ and $B$, one can easily check that it is. 
An example of such an adjacency matrix of a monotone graph is
$$
C = \left(\begin{matrix}
1 & 1 & 1 & 0 & 0 & 0\\
1 & 1 & 1 & 1 & 0 & 0\\
0 & 1 & 1 & 1 & 1 & 0\\
0 & 0 & 1 & 1 & 1 & 0\\
0 & 0 & 1 & 1 & 1 & 1\\
0 & 0 & 1 & 1 & 1 & 1
\end{matrix}\right).
$$
Moreover, we say $G$ is a $\gamma$-dense monotone graph if every vertex in $A \cup B$ has degree at least $\gamma n$.
\end{definition}

The main theorem of this section is given below.
\begin{theorem}\label{thm:monotone}
Let $\mathcal{D}$ be the set of all monotone graphs with $\delta(G) \geq n/2$.\footnote{Remember that the total number of nodes in $G$ is $2n$ in the bipartite case.} If
$\mathcal{D}$ is strongly $k$-switch irreducible for Hamiltonian cycles for some $k \in \mathbb{N}$ (this is the case for $k = 10$ by Remark~\ref{rem:bipartite})
 then for every $G \in \mathcal{D}$, the $k$-switch Markov chain for sampling a Hamiltonian cycle from $\mathcal{H}_G$ is rapidly mixing. 
\end{theorem}

As mentioned earlier, the set of all Hamiltonian cycles for (not necessarily dense) monotone graphs is connected under switches of size two \cite{Takaoka2018} in the weak sense as defined in the preliminaries. Takaoka shows that every Hamiltonian cycle can be transformed into a `canonical' Hamiltonian cycle using switches of size two. This is, however, not enough for the argument we will give below. For that we need the strong sense of irreducibility.

\begin{proof}[Proof of Theorem \ref{thm:monotone}]
The proof relies on an embedding argument similar to that in \cite{Feder2006}, but technically somewhat different.
While the argument in  \cite{Feder2006}  corresponds to the case where $G$ is a complete bipartite graph (which is indeed monotone), here we relax the argument so that it extends to monotone graphs.

 Let $G \in \mathcal{D}$ be given. In particular, our goal is to show, for every $G \in \mathcal{D}$, the existence of a function 
$
\phi: \mathcal{F}_G \rightarrow \mathcal{H}_G 
$
with the properties
\begin{enumerate}[i)]
\item $\phi^{-1}(H) \leq \text{poly}(n)$ for every $H \in \mathcal{H}_G$, and,
\item there exists a function $f : \N \rightarrow \N$ such that whenever $F,F' \in \mathcal{F}_G$ with $|F \triangle F'| \leq k$, we have $|\phi(F)\triangle\phi(F')| \leq f(k)$.
\end{enumerate}
If such a function exists, one can argue exactly as in \cite{Feder2006} that every efficient multi-commodity flow for the $k$-switch Markov chain on the set of all $2$-factors $\mathcal{F}_G$ can be transformed into an efficient multi-commodity flow for the $k$-switch Markov chain on the set of all Hamiltonian cycles $\mathcal{H}_G$.\footnote{In \cite{Feder2006}, it is shown that any efficient flow for the $2$-switch Markov chain for sampling subgraphs of $K_n$ with a given degree sequence can be turned into an efficient flow for the $2$-switch Markov chain for sampling \emph{connected} graphs with a given degree sequence.} (The embedding argument from \cite{Feder2006} that we refer to here is essentially the same as that used to prove Theorem \ref{thm:2_mixing} in Appendix \ref{app:2_mixing}.) As we know that there exists an efficient multi-commodity flow for the $k$-switch Markov chain 
(by Theorem \ref{thm:2_mixing} and Proposition~\ref{pr:2-factor}), 
this then shows that the $k$-switch Markov chain on $\mathcal{H}_G$ is also rapidly mixing. \\

The remainder of the proof is dedicated to showing the existence of such a function $\phi$ for each $G \in \mathcal{D}$, which we will do in three claims. Let $G = (A \cup B, E) \in \mathcal{D}$ be a monotone graph with $|A|=|B|=n$ where we assume that $n$ is even for simplicity.\footnote{When $n$ is odd, one can work with $\lceil n/2 \rceil$ instead of $n/2$ throughout the proof.} Let $a_1, \ldots, a_n$ (resp.\ $b_1, \ldots, b_n$) be the vertices of $A$ (resp.\ $B$) in order as given in Definition~\ref{def:monotone}.  Set $A_1 = \{a_1, \ldots, a_{n/2} \}$ with $A_2 = A \setminus A_1$ and $B_1 = \{b_1, \ldots, b_{n/2}\}$ with $B_2 = B \setminus B_1$.

\begin{claim}
\label{cl:1}
With the setup above, the graphs $G[A_1 \cup B_1]$ and $G[A_2 \cup B_2]$ are complete bipartite. 
\end{claim}
  
\begin{claim}
\label{cl:2}
Given $G \in \mathcal{D}$, let $\mathcal{P}_G$ be the set of all subgraphs  $K \subseteq G$ such that $K$ is the union of three vertex-disjoint paths that together cover all vertices of $G$. Then there exists an injective function $\phi_1: \mathcal{F}_G \rightarrow \mathcal{P}_G$ and a function $g:\mathbb{N} \rightarrow \mathbb{N}$ such that whenever $F,F' \in \mathcal{F}_G$ with $|F \triangle F'| \leq k$, we have $|\phi_1(F) \triangle \phi_1(F')| \leq g(k)$.
\end{claim}  
  
\begin{claim}
\label{cl:3}
Given $G \in \mathcal{D}$, there is a function $\phi_2: \mathcal{P}_G \rightarrow \mathcal{H}_G$ such that for 
every $K \in \mathcal{P}_G$, we have that $|K \triangle \phi(K)| \leq 9$; in particular, for 
each $H \in \mathcal{H}_G$, $|\phi_2^{-1}(H)| \leq |E(G)|^9 = {\rm poly}(n)$. 
\end{claim}

The function $\phi$ is the composition of $\phi_1$ and $\phi_2$ and can easily be seen to satisfy the desired properties (taking $f(k) = g(k) + 18$). Therefore it remains only to prove the claims. 

\begin{proof}[Proof of Claim~\ref{cl:1}]
Note that $a_1b_1$ must be an edge of $G$. If this is not the case, then $b_1$ can never have positive degree, because of monotonicity of the rows of the adjacency matrix. As both $a_1$ and $b_1$ have degree at least $n/2$, we can conclude that all edges of the form $a_ib_j$ with $1 \leq i,j \leq n/2$ are present (again because of monotonicity) so $G[A_1 \cup B_1]$ is complete bipartite. A similar argument holds for the edge $a_nb_n$ that yields $G[A_2 \cup B_2]$ is complete bipartite.
\end{proof}

\begin{proof}[Proof of Claim~\ref{cl:2}]
We use a similar idea as in \cite{Feder2006}.  We fix the total orderings 
$$
a_{\frac{n}{2}+1} < a_{\frac{n}{2}+2} < \dots < a_n < a_1 < a_2 < \dots <a_{\frac{n}{2}}
$$ on the vertices in $A$ and 
$$
b_{\frac{n}{2}+1} < b_{\frac{n}{2}+2} < \dots < b_n < b_1 < b_2 < \dots <b_{\frac{n}{2}}
$$ 
on the vertices of $B$.

Fix $F \in \mathcal{F}_G$ and let $C_1,\dots,C_q$ be the cycles (or connected components) of $F$. For a given cycle $C_r$, we use $a^r$ to denote the highest ordered vertex of $A$ in $C_r$, and we use $b^r$ to denote the highest ordered vertex of $B$ in $C_r$. We first group the cycles in three sets depending on the vertices $a^r$ and $b^r$. We define
$$
Q_{A_1} = \{C_r:  a^r \in A_1\}, \ \ \ \ \ \ Q_{B_1} = \{C_r : a^r \in A_2 \text{ and }  b^r \in B_1 \}
$$
and $Q_{A_2 \cup B_2}$ as the set of all remaining cycles not in $Q_{A_1}$ or $Q_{B_1}$. Note that the cycles in $Q_{A_2 \cup B_2}$ are fully contained in $A_2 \cup B_2$. For each  cycle $C^r$ in $Q_{A_1}$ and $Q_{A_2 \cup B_2}$, let $c^r$ be an arbitrary neighbour of $a^r$ in $C^r$ and for each cycle $C^r$ in $Q_{B_1}$ let $d^r$ be an arbitrary neighbour of $b^r$ on $C^r$ (in each case there are two choices). We delete the edges $a^rc^r$ and $b^rd^r$ from $F$ to create paths; we will connect the paths in each group together to build the three paths which will define $\phi_1(F) \in \mathcal{P}_G$.

We first explain the idea (of Feder et al.\ \cite{Feder2006}) on how to glue together the paths from $Q_{A_2 \cup B_2}$ in such a way that we can uniquely recover the original paths from the single glued path: this case is easiest because we know from Claim~\ref{cl:1} that the graph $G[A_2 \cup B_2]$ is complete bipartite.

After renaming the cycles, let us assume the cycles in $Q_{A_2 \cup B_2}$ are $C^1, \ldots, C^q$ where $a^1 < a^2 \dots < a^q$.  Let $P_r$ be the path obtained by 
deleting the edge $a^rc^r$ from the cycle $C_r$. As all the cycles lie entirely within $A_2 \cup B_2$ and $G[A_2 \cup B_2]$ is complete bipartite, we know that all the edges $c^ra^{r+1}$ are present in $G$ for $r = 1,\dots,q-1$. Adding 
these edges to the graph consisting of $P_1,\dots,P_q$, results in a path that we call $P_{A_2 \cup B_2}$. 
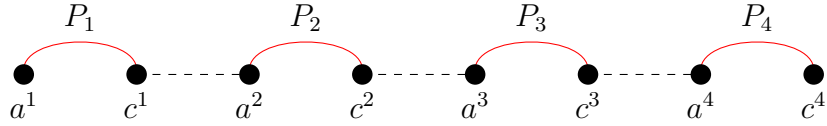
\begin{figure}[h!]
\centering
\begin{tikzpicture}[scale=1.5]
\coordinate (A1) at (0,0); 
\coordinate (A2) at (1,0);
\coordinate (M1) at (2,0);
\coordinate (M2) at (3,0);
\coordinate (N1) at (4,0);
\coordinate (N2) at (5,0);
\coordinate (T1) at (6,0);
\coordinate (T2) at (7,0);

\path[every node/.style={sloped,anchor=south,auto=false}]
(A1) edge[-,red,bend left=100] node {\textcolor{black}{$P_1$}} (A2)
(A2) edge[-,dashed] node {} (M1)
(M1) edge[-,red,bend left=100] node {\textcolor{black}{$P_2$}} (M2)
(M2) edge[-,dashed] node {} (N1)
(N1) edge[-,red,bend left=100] node {\textcolor{black}{$P_3$}} (N2)
(N2) edge[-,dashed] node {} (T1)
(T1) edge[-,red,bend left=100] node {\textcolor{black}{$P_4$}} (T2);  

\node at (A1) [circle,scale=0.7,fill=black] {};
\node (a1) [below=0.1cm of A1]  {$a^1$};
\node at (A2) [circle,scale=0.7,fill=black] {};
\node (a2) [below=0.1cm of A2]  {$c^1$};
\node at (M1) [circle,scale=0.7,fill=black] {};
\node (m1) [below=0.1cm of M1]  {$a^2$};
\node at (M2) [circle,scale=0.7,fill=black] {};
\node (m2) [below=0.1cm of M2]  {$c^2$};
\node at (N1) [circle,scale=0.7,fill=black] {};
\node (n1) [below=0.1cm of N1]  {$a^3$};
\node at (N2) [circle,scale=0.7,fill=black] {};
\node (n2) [below=0.1cm of N2]  {$c^3$};
\node at (T1) [circle,scale=0.7,fill=black] {};
\node (t1) [below=0.1cm of T1]  {$a^4$};
\node at (T2) [circle,scale=0.7,fill=black] {};
\node (t2) [below=0.1cm of T2]  {$c^4$};

\end{tikzpicture}
\caption{Sketch of path $P$ from the paths $P_1,\dots,P_q$ for the case $q = 4$.}
\label{fig:case2}
\end{figure} 

Note that, given $P_{A_2 \cup B_2}$, (without knowing the paths $P_1, \ldots, P_q$), we can uniquely recover these $P_1,\dots,P_q$ as follows. We know that the endpoint of $P_{A_2 \cup B_2}$ that is contained in $A$ is the first vertex of $P_1$, i.e., the vertex $a^1$ (the other endpoint is necessarily in $B$). In order to recover $P_1$ we start following the path $P_{A_2 \cup B_2}$, starting from $a^1$, until we reach the first vertex in $A$ that is ordered higher than $a^1$; this is the first vertex of $P_2$, i.e., the vertex $a^2$. 
Continuing in this fashion we can uniquely recover all the paths $P_i$.

We apply a similar procedure to the paths obtained from $Q_{A_1}$ and $Q_{B_1}$ to form paths $P_{A_1}$ and $P_{B_1}$, respectively. The problem here is that the underlying graph is  not complete bipartite so we do not apriori know if the edges to `glue' the paths together are all present: we  argue that they are in fact present. 
The proof for $Q_{A_1}$ that we will give below also holds for $Q_{B_1}$ by symmetry of monotonicity (the case of $Q_{B_1}$ is essentially a slightly more restrictive setting in which some of the cases below cannot occur).

Assume that the cycles in $Q_{A_1}$ are $C_1,\dots,C_p$ labelled so that $a^1 < a^2 < \dots < a^p$. By means of a case distinction, depending on whether $c^r \in B_1$ or $c^r \in B_2$ for $r = 1,\dots,p-1$, we will show that the edges $c^ra^{r+1}$ always exist. 

\textit{Case 1: $c^r \in B_1$.} As we know that $a^{r+1} \in A_1$, by definition of $Q_{A_1}$ it follows that $c^ra^{r+1}$ is in $G$, since $G[A_1 \cup B_1]$ is complete bipartite by Claim~\ref{cl:1}.

\textit{Case 2: $c^r \in B_2$.}  Since $a^r < a^{r+1} =: a_j$ by assumption, monotonicity tells us that the neighbourhood $N(a^{r+1}) \subseteq B$ ends at either $c^r$ or to the right of $c^r$. 
Furthermore, we know $a_jb_j \in E(G)$, again since $G[A_1 \cup B_1]$ is bipartite by Claim~\ref{cl:1}. Since $b_j \in B_1$, it lies to the left of $c^r \in B_2$ so, in particular, the neighbourhood $N(a^{r+1})$ starts before $c^r$. Monotonicity then tells us that the edge $c^ra^{r+1}$ is also present in $G$.

We have shown how to construct the paths $P_{A_1}$, $P_{B_1}$, and $P_{A_2 \cup B_2}$, which together clearly cover all vertices of $G$. We define $\phi_1(F) = P_{A_1} \cup P_{B_1} \cup P_{A_2 \cup B_2} \in \mathcal{P}_G$.

In order to see that $\phi_1$ is injective, note first that if $K \in \mathcal{P}_G$ is the image of some (unknown ) $F \in \mathcal{F}_G$ under $\phi_1$, then one of the paths in $K$ has all its vertices in $A_2 \cup B_2$ (we call this path $P_{A_2 \cup B_2}$), one has all its vertices from $A$ in $A_2$ and some vertices from $B_1$ (we call this path $P_{B_1}$), and we call the remaining path $P_{A_1}$. As described earlier, we can then easily identify the constituent paths that were glued together to form $P_{A_1}$, $P_{B_1}$, and $P_{A_2 \cup B_2}$.
Finally we can complete each constituent path to a cycle to uniquely recover $F$. Therefore $\phi_1$ is injective.

Finally, suppose $F, F' \in \mathcal{F}_G$ with $|F \triangle F'| \leq k$. In particular, there are at most $k$ cycles that belong to  one of $F$ or $F'$ but not both.
  In constructing $\phi_1(F)$ (resp.\ $\phi_1(F')$), we first delete one edge from each cycle of $F$ (resp.\ $F'$) to obtain a union of paths, which we call $J$ (resp.\ $J'$). Then $|J \triangle J'| \leq k$ and there are at most $k$ paths that belong to one of $J$ or $J'$ but not both. When gluing paths of $J$ (resp.\ $J'$) together to form $\phi_1(F)$ (resp.\ $\phi_1(F')$) there are at most $2k$ gluing edges that are used for one of $J$ or $J'$ but not both (at most two such edges for each differing path).  This shows that $|\phi_1(F) \triangle \phi_1(F')| \leq k + 2k = 3k$, showing $\phi_1$ has the desired property (taking $g(k)=3k$).
\end{proof}

\begin{proof}[Proof of Claim~\ref{cl:3}]
This claim follows immediately from Lemma~\ref{lem:recon} below.
\end{proof}

\begin{lemma}
\label{lem:recon}
Suppose $G=(V,E)$ is an $n$-vertex graph with $\delta(G) >  n/2$. If $P_1, \ldots, P_k$ are $k$ vertex-disjoint paths in $G$ that together cover all vertices $V$, then there exists a Hamiltonian cycle $H$ of $G$ such that $E(H) \triangle E(P_1 \cup \cdots \cup P_k) \leq 3k$. 

For bipartite graphs, we have the following. Suppose $G=(V,E)$ is a bipartite graph with bipartition $V = A \cup B$ with $|A|=|B|=n$ and $\delta(G) \geq n/2$. If $P_1, \ldots, P_k$ are $k$ vertex-disjoint paths in $G$ that together cover all vertices $V$, then there exists a Hamiltonian cycle $H$ of $G$ such that $E(H) \triangle E(P_1 \cup \cdots \cup P_k) \leq 3k$.
\end{lemma}
We prove the lemma for graphs; an almost identical proof works for bipartite graphs and we indicate where the proofs differ.
\begin{proof}
We will inductively modify the system of paths, at each step modifying at most $3$ edges and reducing the number of paths by $1$.

Let $x_i$ and $y_i$ be the endpoints of $P_i$ and orient the path $P_i$ from $x_i$ to $y_i$. For any vertex $x$, let $x^+$ (resp.\ $x^-$) be the successor (resp.\ predecessor) of $x$ on its path (note that these exist except possibly at the $2k$ endpoints of the paths). For any set $S \subseteq V(G)$, we define $S^+:=\{x^+: s \in S \}$.

Assuming $k \geq 2$, take any two paths, say $P_1$ and $P_2$. [If $G$ is bipartite, we choose $P_2$ s.t. $x_1$ and $y_2$ are in different parts, say $x_1\in A$ and $x_2\in B$.  Note that this is always possible, renaming paths if necessary.] 
If $x_1$ is adjacent to any of $x_2, \ldots, x_k$, say to $x_i$, then we can reduce the number of paths by replacing $P_1$ and $P_i$ by $y_1P_1x_1x_iP_iy_i$ as required (only modifying one edge) and we continue. Therefore we may assume that $x_1$ is not adjacent to any of $x_2, \ldots, x_k$, and in particular, $|N(x_1)^-| = |N(x_1)| > n/2$. Then since $|N(y_2)| > n/2$, we must have that $N(x_1)^- \cap N(y_2)$ is non-empty. [Note that for $G$ bipartite $N(x_1)^-,N(y_2)\subseteq A$ and therefore $N(x_1)^- \cap N(y_2)$ also holds.] Let $z \in N(x_1)^- \cap N(y_2)$ and assume $z \in V(P_i)$ for some $i = 1, \ldots, k$. If $i \not = 1,2$ then we can replace $P_1, P_2, P_i$ with the two paths $y_1P_1x_1z^+P_iy_i$ and $x_iP_izy_2P_2x_2$, which together cover all the vertices of $V(P_1) \cup V(P_2) \cup V(P_i)$ (see Figure~\ref{fig:path2hc} (a)). If $i=1$, we replace $P_1, P_2$ with the path $y_1P_1z^+x_1P_1zy_2P_2x_2$ (see Figure~\ref{fig:path2hc} (b)) and if $i=2$, we replace  $P_1, P_2$ with $y_1P_1x_1z^+P_2y_2zP_2x_2$. In all three of these cases, we delete one edge and add two (i.e.\ we modify three edges) and reduce the number of paths by $1$. 

\begin{figure}[ht!]
	\begin{center}
		\includegraphics[width=0.9\linewidth]{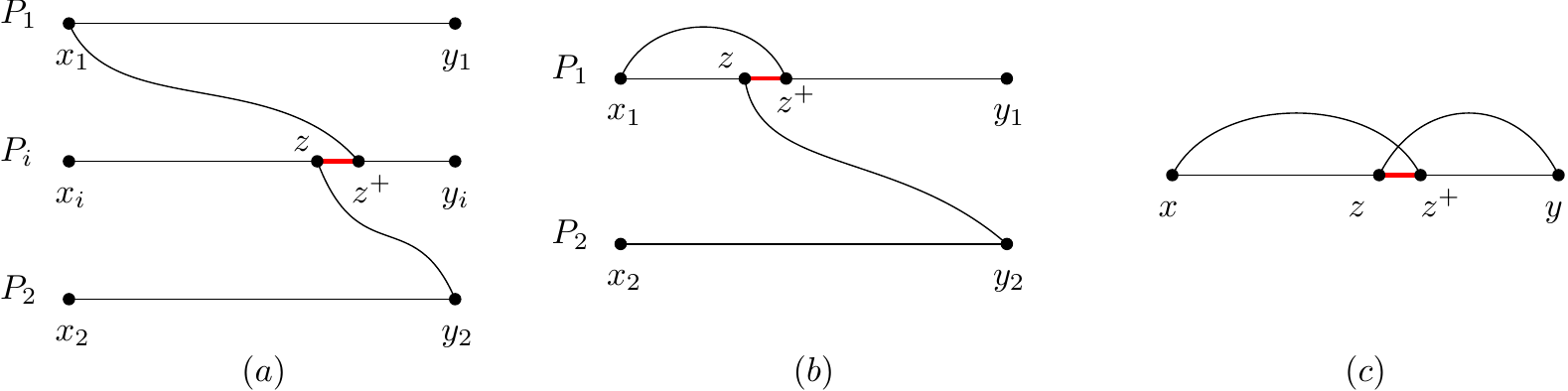}
		\caption{(a) and (b): Reducing the number of paths, cases $i\neq 1,2$ and $i=1$. Case $i=2$ is similar. (c): completing the Hamiltonian cycle. In all cases, the thick, red edge is removed, and the curvy edges are introduced.}
		\label{fig:path2hc}
	\end{center}
\end{figure}

By iterating this, we obtain a Hamiltonian path $P$ by modifying at most $3(k-1)$ edges. We can then complete this to a Hamiltonian cycle in the standard way. Let $x$ and $y$ be the endpoints of $P$ and pick $z \in N(x)^- \cap N(y)$ (which exists as before since $|N(x)^-|, |N(y)| > n/2$). Then we obtain a Hamiltonian cycle $H = xPzyPz^+x$ (see Figure~\ref{fig:path2hc} (c)), where again we have added two edges and removed one. [In the case of $G$ being bipartite, $P$ has its endpoints in different parts, so that again $N(x)^-$, $N(y)$ are subsets of the same part, so again $N(x)^- \cap N(y) \not= \emptyset$.]
\end{proof}

This completes the proof of the three claims and hence of the theorem.
\end{proof}

\subsection{Remarks regarding the density assumption}
It is perhaps interesting to note, in general, it is necessary to make some kind of assumption on the minimum degree of the monotone graph  for the argument in the proof of Theorem \ref{thm:monotone} to work. Without it, it is not necessarily true that the number of $2$-factors is at most a polynomial factor larger than the number of Hamiltonian cycles of a given graph $G$. See the matrix below for an indication of the family of instances that should achieve this claim. 
$$
\left(\begin{matrix}
1 & 1 & 0 & 0 & 0 & 0 \\
1 & 1 & 1 & 0 & 0 & 0 \\
1 & 1 & 1 & 1 & 0 & 0 \\
1 & 1 & 1 & 1 & 1 & 0 \\
1 & 1 & 1 & 1 & 1 & 1 \\
1 & 1 & 1 & 1 & 1 & 1 
\end{matrix}\right)
$$
We next explain why this claim is true. Let the rows be indexed by $A = (a_1,\dots,a_n)$ and the columns by $B = (b_1,\dots,b_n)$. As $a_1$ only has two neighbours, any Hamiltonian cycle must contain the edges $a_1b_1$ and $a_1b_2$. This is indicated in the matrix below.
$$
\left(\begin{matrix}
\mathbf{\underline{1}} & \mathbf{\underline{1}} & 0 & 0 & 0 & 0 \\
1 & 1 & 1 & 0 & 0 & 0 \\
1 & 1 & 1 & 1 & 0 & 0 \\
1 & 1 & 1 & 1 & 1 & 0 \\
1 & 1 & 1 & 1 & 1 & 1 \\
1 & 1 & 1 & 1 & 1 & 1 
\end{matrix}\right)
$$
Now, the vertex $a_2$ cannot also have neighbours $b_1$ and $b_2$, as this creates a cycle of length four. So we have $N(a_2) = \{b_1,b_3\}$ or $N(a_2) = \{b_2,b_3\}$; see the matrices below. 
$$
\left(\begin{matrix}
\mathbf{\underline{1}} & \mathbf{\underline{1}} & 0 & 0 & 0 & 0 \\
\mathbf{\underline{1}}  & 1 & \mathbf{\underline{1}}  & 0 & 0 & 0 \\
1 & 1 & 1 & 1 & 0 & 0 \\
1 & 1 & 1 & 1 & 1 & 0 \\
1 & 1 & 1 & 1 & 1 & 1 \\
1 & 1 & 1 & 1 & 1 & 1 
\end{matrix}\right) \quad \text{ or } \quad
\left(\begin{matrix}
\mathbf{\underline{1}} & \mathbf{\underline{1}} & 0 & 0 & 0 & 0 \\
1 & \mathbf{\underline{1}}  & \mathbf{\underline{1}}  & 0 & 0 & 0 \\
1 & 1 & 1 & 1 & 0 & 0 \\
1 & 1 & 1 & 1 & 1 & 0 \\
1 & 1 & 1 & 1 & 1 & 1 \\
1 & 1 & 1 & 1 & 1 & 1 
\end{matrix}\right)
$$
Note that in both the matrices above, there is now one vertex in $B$ that has two neighbours already (and therefore cannot be chosen as neighbour in any later step). By repeating this argument, one can show that for every row $i = 2,\dots,n-1$ there are two possible choices of extending the current Hamiltonian path, and so the number of Hamiltonian cycles equals $2^{n-2}$.

However, the number of $2$-factors in at least $(n/4)!$. To see this, first note that this is a lower bound on the number of Hamiltonian cycles in the (complete) subgraph induced by the vertices $\{a_{3n/4+1},\dots,a_n\}$ and $\{b_1,\dots,b_{n/4}\}$ (assuming that $n$ is divisible by four). It is not hard to see that any Hamiltonian cycle on this induced subgraph can be extended to a $2$-factor of the original bipartite graph.\footnote{One can give a sharper bound here than $(n/4)!$, but this is not needed for our purposes.} \\

Nevertheless, we believe that our result can be generalized to monotone graphs with minimum degree $\gamma n$ for any $\gamma  \in (0,1)$. However, this comes at the expense of many more technicalities that (in our opinion) do not offer any additional insights. Remember that in Claim \ref{cl:1}, we show that the nodes of $G$ can be partitioned into two complete bipartite graphs whenever $\gamma \geq 1/2$. More generally, for a given $\gamma \in (0,1)$, it should be possible to partition the nodes of $G$ into a constant $c = c(\gamma)$ number of complete bipartite graphs. The analogue of Claim \ref{cl:2} would then be to show that all cycles in a given $2$-factor can be broken up, and glued together again, into a constant $d(\gamma)$ number of (vertex-disjoint) paths, after which one would need to argue that the resulting collection of paths is close, in terms of symmetric difference, to a Hamiltonian cycle in the monotone graph.  

 \bibliographystyle{abbrv}
 \bibliography{references}

\newpage
\appendix
\section{Rapid mixing of switch-based chains for sampling $2$-factors in dense graphs}\label{app:2_mixing}
\emph{This section is a modification of certain parts in \cite{AK2019}.}\\

\noindent We will tailor all definitions to the notion of $2$-factors for sake of readibility.
Let $\mathbf{2}= (2,2,\dots,2)$ be the all-twos sequence of length $n$. Let $G$ be a given (dense) undirected graph $G$ and let $\mathcal{F}_G$ be the set of all $2$-factors of $G$.

We write $G(d')$ for the set of all subgraphs of $G$ 
with degree sequence $d'$.
Let $\mathcal{F}'_G= \cup_{d'} G(d')$ with $d'$ ranging over the set 
$$\left\{d'  : d'_j \leq 2 \text{ for all } j \text{, and } \sum_{i=1}^n |2 - d_i'| \leq 2\right\}.
$$
In other words, $\mathcal{F}'_G$ is the set of almost $2$-factors, that is, subgraphs of $G$ with degree sequence $d'$ where (i) $d' = \mathbf{2}$, or (ii) there exist distinct $\kappa,\lambda$ such that
$d'_i = 1$ if $i \in \{\kappa,\lambda\}$ and $d'_i = 2$ otherwise,
or (iii) there exists a $\kappa$ so that
$d'_i = 0$ if $i = \kappa$ and $d'_i = 2$ otherwise. 
In the case (ii) we say that $d'$ has two vertices with degree deficit one, and in the case (iii) we say that $d'$ has one vertex with degree deficit two. 

A family $\mathcal{D}$ of graphs $G$ is called \emph{$P$-stable} \cite{Jerrum1990} if there exists a polynomial $q(n)$ such that for all $G \in \mathcal{D}$ 
we have $|\mathcal{F}_G'|/|\mathcal{F}_G| \leq q(n)$  where $n$ is the number of vertices of $G$.\\ 

\noindent Jerrum and Sinclair \cite{Jerrum1990} define a Markov chain that, tailored to $2$-factors, works as follows.\medskip

\setlength{\fboxsep}{5pt}
\noindent \fbox{\parbox{\textwidth-12pt}{Let $F \in \mathcal{F}'_G$ be the current $2$-factor of the JS chain. Choose an ordered pair of vertices $(i,j)$ uniformly at random:
\begin{enumerate}
	\item if $F \in \mathcal{F}_G$ and $(i,j)$ is an edge of $F$, delete $(i,j)$ from $G$ (\emph{Type 0 transition}),
	\item if $F \notin \mathcal{F}_G$ and the degree of $i$ in $G$ is less than $2$, and $(i,j)$ is not an edge of $F$, add $(i,j)$ to $F$ if this edge is in $G$; if this causes the degree of $j$ to exceed $2$, select an edge $(j,k)$ uniformly at random from $F$ and delete it (\emph{Type 1 transition}).
	\end{enumerate}

In case the degree of $j$ does not exceed $2$ in the second case, we call this a \emph{Type 2 transition}.}}
\\

\bigskip

The graphs $F,F' \in \mathcal{F}'_G$ are  
\emph{JS adjacent} if $F$ can be obtained from $F'$ with positive probability in one transition of the JS chain and note this relation is symmetric. The properties of the JS chain, stated in Theorem \ref{thm:js_properties} below, are easy to check \cite{Jerrum1990}.
\begin{theorem} \label{thm:js_properties}
	The JS chain on $\mathcal{F}'_G$ is irreducible, aperiodic and symmetric, and, hence, has uniform stationary distribution over $\mathcal{F}'_G$. Moreover, $P(F,F')^{-1} \leq 2n^3$ for all JS adjacent $F,F' \in \mathcal{F}'_G$, and also 
	the maximum in- and out-degrees of the state space graph of the JS chain are bounded by $n^3$.
\end{theorem}

We say that two graphs $F, F' \in \mathcal{F}'_G$ are \emph{within distance $r$ in the JS chain} if there exists a path of length  at most $r$ from $F$ to $F'$ in the state space graph of the JS chain. By $\text{dist}(F',\mathbf{2})$ we denote the  minimum distance of $F' \in \mathcal{F}'_G$ to an element in $\mathcal{F}$.
The following parameter will play a central role in this work.  Let 
\begin{equation}\label{eq:distance}
k_{JS}(G) = \max_{F' \in \mathcal{F}'_G} \text{dist}(F',\mathbf{2}) \,.
\end{equation}
Based on the parameter $k_{JS}$, we define the notion of \emph{strong stability} \cite{AK2019}. 

\begin{definition}[Strong stability]\label{def:strong_stable}
	A family of graphs $\mathcal{D}$ is called \emph{strongly stable} if there exists a constant $\ell$ such that $k_{JS}(G) \leq \ell$ for all $G \in \mathcal{D}$.
\end{definition}

It is shown by Jerrum and Sinclair \cite{Jerrum1990}, that if $\mathcal{D}$ is the set of all graphs $G$ with $\delta(G) \geq n/2$, then $\mathcal{D}$ is strongly stable for $\ell = 3$. (This gives rise to the condition on the minimum degree in the statement of Theorem \ref{thm:2_mixing}.) \\

\noindent We now have all the ingredients for the proof of Theorem \ref{thm:2_mixing}. It uses essentially the same argument as that in \cite{AK2019}, where it is shown that the switch Markov chain for sampling graphs with given degrees is rapidly mixing for certain strongly stable classes of degree sequence, i.e., for the notion of strong stability in that setting.

\begin{proof}[Proof of Theorem \ref{thm:2_mixing}] 
The high-level idea is to use an embedding argument which states that an efficient multi-commodity flow for the JS chain can be transformed into an efficient flow for the $k$-switch Markov chain.\\ 

\noindent The fact that there exists an efficient multi-commodity flow for the JS chain can be shown using exactly the same arguments as in Theorem 3.2 in \cite{AK2019}.\footnote{That theorem essentially shows the result in the case where the graph $G$ is complete and strong irreducibility for $k  = 2$, but the analysis remains true when $G$ is not a complete graph, and when $k > 2$ (still assuming the notion of strong stability of  the given class of degree sequences).}

Without going into all the details, we will give a sketch of this argument. Recall that Sinclair's multi-commodity flow method asks us to define a flow $f$ in the state space graph of the JS chain that routes a fraction $\pi(X)\pi(Y)$ of flow from $X$ to $Y$ for every $X, Y \in \mathcal{F}_G'$.  Here,  
$$
\pi(Z) = \frac{1}{|\mathcal{F}_G'|}
$$ 
for every $Z \in \mathcal{F}_G'$. 

The notion of strong stability allows us to take a shortcut here: Instead of defining a flow between every two states in $\mathcal{F}_G'$, one can first define a flow between any two $2$-factors $F, F' \in \mathcal{F}_G$. Then, roughly speaking, in order to define a flow between any two states in $\mathcal{F}_G'$, we use the fact that every `almost $2$-factor' $X \in \mathcal{F}_G' \setminus \mathcal{F}_G$ is close to some actual $2$-factor in the state space graph, because of strong stability. These short paths between states in  $\mathcal{F}_G' \setminus \mathcal{F}_G$ and $\mathcal{F}_G$ can be exploited to define the desired flow between any two states in $\mathcal{F}_G'$.

In order to define the flow between two $2$-factors $F$ and $F'$, we decompose the symmetric difference $F \triangle F'$ into a collection of alternating circuits.\footnote{To be more precise, the flow is spread out over all possible ways in which the symmetric difference can be decomposed.} We then use the operations defining the JS chain in order to transform $F$ into $F'$ by `flipping'  edges on an alternating circuit in order to move from $F$ to $F'$ (see Figures 4--6 in \cite{AK2019} for an example). In particular, all these flow-carrying paths will have polynomial length. Morever, all these operations only use edges in $F \triangle F'$ and so the approach taken in the proof of Theorem 3.2 in \cite{AK2019} can be used here as well (when $G$ is not a complete graph). In order to prove that this procedure indeed yields an efficient flow, one can use the exact same arguments as in \cite{AK2019}. 

In particular, we can obtain the following statement similar to Theorem 3.2 in \cite{AK2019}.

\begin{lemma}\label{lem:js_flow}
Let $\mathcal{D}$ be the collection of graphs with $\delta(G) \geq n/2$. Then
there exist polynomials $p(n)$ and $r(n)$ such that for any $G \in \mathcal{D}$ there exists an efficient multi-commodity flow $f$ for the JS chain on $\mathcal{F}_G'$ satisfying
$$
\max_e f(e) \leq p(n) \ \text{ and } \ \ell(f) \leq q(n).
$$
where $f(e)$ is the total amount of flow routed over edge $e$ in the state space graph, and $\ell(f)$ the maximum length of a flow-carrying path.
\end{lemma}\medskip

\noindent The next step entails transforming the flow $f$ in Lemma \ref{lem:js_flow} into an efficient multi-commodity flow for the $k$-switch Markov chain (assuming strong irreducibility). First note that the flow $f$ above is a flow between any two states in $\mathcal{F}_G'$, whereas we are interested in defining a flow, let us call it $g$, between any two states in $\mathcal{F}_G$. Therefore, the first step will be to restrict ourselves to the flow routed in $f$ between states in $\mathcal{F}_G$, which we call $\tilde{f}$. 

A subtlety here is that we route a flow of $1/|\mathcal{F}_G'|^2$ between any two states in $\mathcal{F}_G$ in $\tilde{f}$ (and also $f$), whereas we need to route $1/|\mathcal{F}_G|^2$ between two such states in the desired (final) flow $g$. This is not a problem as replacing $|\mathcal{F}_G'|$ by $|\mathcal{F}_G|$ in the definition of $\tilde{f}$ only blows up the congestion on a given edge $e$ by at most a polynomial factor, using the fact that
$$
\frac{|\mathcal{F}_G'|}{|\mathcal{F}_G|} \leq s(n)
$$
for some polynomial $s$, since $\gamma \geq 1/2$.\footnote{Given $F \in \mathcal{F}'_G \setminus \mathcal{F}_G$, let $x,y$ be vertices of degree $1$ or $x=y$ the vertex of degree $0$. Find $z \in N(x) \cap N(y)^+$ and replace $zz^-$  with $xz$, $yz^-$ to obtain $\sigma(F) \in \mathcal{F}_G$ with $|F \triangle \sigma(F)| \leq 3$. Thus $|\sigma^{-1}(F)| \leq n^3 =: s(n)$.} Let us call the resulting (intermediate) flow $\bar{f}$, which now routes  a fraction $1/|\mathcal{F}_G|^2$ of flow  between any two states in $\mathcal{F}_G$ in the JS chain, and that has polynomially bounded congestion.\footnote{The flows $\tilde{f}$ and $\bar{f}$ are not efficient multi-commodity flows for Markov chains, but `auxiliary flows'.}\\

\noindent We next continue with transforming the flow $\bar{f}$ into the desired flow $g$ (again similar to the ideas in \cite{AK2019}). We do this by a sequence of reductions. 

 We first identify for every $X \in \mathcal{F}'_G\setminus \mathcal{F}_G$ some $2$-factor $\psi(X) \in \mathcal{F}_G$ that is within $k_{JS} = 3$ moves (in the JS chain) away from $X$. All $X$ that map unto the same $2$-factor $F = \psi(X)$ are merged with $F$ into a supervertex that we identify with $F$. If this procedure gives rise to parallel (directed) edges, we replace them by one edge and route all flow over that edge; self-loops are removed. It is not hard to see $|\psi^{-1}(F)|$ has size polynomial in $n$, intuitively, as we only merge vertices that are close to each other (in the original JS chain). Moreover, it is not hard to see that this procedure will only give rise to at most a polynomial number of parallel edges between two given vertices in $\mathcal{F}_G$ (for the same reason). Let us call the resulting (simple) graph $\mathbb{J} = (\mathcal{F}_G,A)$ and the resulting flow in this graph $f^*$. By what is said above, we have $\max_e f^*(e) \leq p'(n)$ for some polynomial $p'$, i.e., the congestion of $f^*$ is at most a polynomial factor larger than that of $\bar{f}$.

The final problem, before we obtain the desired flow $g$, is that the graph $\mathbb{J}$ contains edges (possibly with flow) between $2$-factors $F,F' \in \mathcal{F}_G$ that might be more than a $k$-switch away from each other. Said differently, these edges do not represent transitions in the $k$-switch Markov chain. Let us partition the edge set $A = A_{\text{switch}} \cup A_{\text{infeasible}}$ where $A_{\text{switch}}$ contains all edges of $A$ that represent a transition in the $k$-switch Markov chain, and $ A_{\text{infeasible}}$ all those edges that do not. 

We argue that for every edge $a = (F,F') \in  A_{\text{infeasible}}$, we can always find a short `detour' in the graph $\mathbb{J}$ using only edges in $A_{\text{switch}}$.
To see this, fix some $a \in A_{\text{infeasible}}$. Suppose that $X$ and $Y$ are adjacent in the JS chain and that $F = \psi(X)$ and $F' = \psi(Y)$ (these $X$ and $Y$ exist by existence of the infeasible edge $a$). Since $k_{JS} = 3$, it can be shown that
$$
|F \Delta F'| \leq 12.
$$ 
Intuitively, this follows from the fact that in the JS chain, $F = \psi(X)$ is close to $X$, which is close to $Y$, which is in turn close to $\psi(Y) = F'$. It follows that with at most $12/k$ switches of size at most $k$, that define the shortcut in $\mathbb{J}$ using only edges in $A_{\text{switch}}$,  we can transform $F$ into $F'$. This follows from the assumption of $k$-switch irreducibility.  Since all these detours take place on a 
 `local' level, the congestion of the resulting multi-commodity flow for the $k$-switch Markov chain, that we get from rerouting the flow of infeasible edges over their respective shortcut, increases at most by a polynomial factor on every \emph{fixed} feasible edge in $\mathbb{J}$. That is, for a fixed edge $b = (F_0,F_0') \in A_{\text{switch}}$, the total number of edges $a = (F,F') \in A_{\text{infeasible}}$ that use $b$ in their detour is at most $\text{poly}(n)$, as (roughly speaking) $F_0$ is at most $12/k$ transitions away from $F$ by construction (and $k$ is constant).

This yields the desired flow $g$.
For a precise and detailed outline of this idea, we refer the reader to \cite{AK2019}. 
\end{proof}

\begin{proof}[Proof of Proposition~\ref{pr:2-factor}]  
We claim that given $F_1, F_2, \in \mathcal{F}_G$, there is a $T \in \mathcal{F}_G$ that can be obtained from $F_1$ by a $4$-switch such that $|T \triangle F_2| < |F_1 \triangle F_2|$. Applying this repeatedly proves the proposition, taking $\phi(k) = k$.

Let $F_1, F_2 \in \mathcal{F}_G$. Note that the symmetric difference of $F_1$ and $F_2$ is the vertex-disjoint union of circuits in which edges alternate between $F_1$ and $F_2$ and the circuits visit each vertex zero, one, or two times. If the symmetric difference of $F_1$ and $F_2$ contains such alternating circuits with four or six edges (corresponding to switches of size $2$ or $3$), then switching along such a circuit reduces the symmetric difference, so assume otherwise.
 
In this case it is not hard to see that we can find an $H_1,H_2$-alternating walk
 $P = a_1a_2a_3a_4a_5a_6$ (here the $a_i$ are vertices and $a_1 $ and $a_6$ are distinct)  
 such that $a_1a_2, a_3a_4, a_5a_6$ are edges of $F_1$, and $a_2a_3$, $a_4a_5$ are edges of $F_2$. 
 
 	We try to find vertices $b$ and $c$ that are neighbours on $F_1$ such that $b\in N(a_1)$ and $c \in N(a_6)$. Then the circuit $C:=a_1a_2a_3a_4a_5a_6cba_1$ is a $4$-switch for $F_1$.  We choose $b$ and $c$ as follows. Orient the cycles of $F_1$ arbitrarily. We call the vertex following a vertex $v$ in this orientation $v^+$ and the previous vertex $v^-$. Set $M=\{v^+: v\in N(a_6)\}$ and consider $N(a_1)\cap M$. As both $|N(a_1)|,|M|\geq n/2+7$ we have $|N(a_1)\cap M|\geq 2 \cdot 7 = 14$.\footnote{In the case of bipartite graphs, we note that $a_1a_2a_3a_4a_5a_6$ is a walk in $G$ and so $a_1$ and $a_6$ are in different parts; say $a_1 \in A$ and $a_6 \in B$. Then $N(a_1), M \subseteq B$, so since $|N(a_1)|, |M| \geq \frac{1}{2}n+7$, so $|N(a_1) \cap M| \geq 14$, and we continue as before.}
	 Select $c \in (N(a_1)\cap M)\setminus \{a_i^+,a_i^-, i=1,\dots,6\}$ and set $b=c^-$.
	 For $T$, the $2$-factor produced by switching $F_1$ along $C:=a_1a_2a_3a_4a_5a_6cba_1$, we see that compared to $F_1$, $T$ contains at least two new edges of $F_2$  (namely $a_2a_3, a_4a_5$) but $T$ may have lost one edge of $F_2$ (namely $bc$ if it was in fact an edge of $F_2$), giving a net gain of one. Since $T$ and $F_1$ have the same number of edges, we see that $|T \triangle F_2| \leq |F_1 \triangle F_2| - 1$, as required.
\end{proof}

\end{document}